\documentclass[10pt]{article}

\usepackage{a4wide}
\usepackage[english]{babel}
\usepackage{amssymb}
\usepackage{amsmath}
\usepackage{amsthm}
\usepackage{array}
\usepackage{graphicx}
\usepackage{verbatim}
\usepackage{url}
\usepackage{float}
\usepackage{framed}
\usepackage{ifthen}
\usepackage[usenames,dvipsnames]{color}
\usepackage[linkcolor=black, linkbordercolor={1 1 1}, citecolor=black, hyperindex, colorlinks=true, bookmarks=true, bookmarksnumbered=true]{hyperref}

\renewcommand{\qed}{\hfill\ensuremath{\Box}}
\renewcommand{\phi}{\varphi}
\renewcommand{\theta}{\vartheta}
\renewcommand{\epsilon}{\varepsilon}

\newtheorem{theorem}[equation]{Theorem}
\newtheorem{lemma}[equation]{Lemma}

\newtheorem{corollary}[equation]{Corollary}

\theoremstyle{definition}

\renewenvironment{proof}[1][Proof]{\begin{trivlist}\item[\hskip \labelsep {\bfseries #1}]}{\qed\end{trivlist}}

\makeatletter
\newenvironment{myleftbar}{%
  \MakeFramed {\advance\hsize-\width}
}{%
  \endMakeFramed%
}
\makeatother
\newtheoremstyle{example}
  {\topsep} {\topsep}%
  {\upshape}
  {}
  {\bfseries}
  {.}
  {\parindent}
  {\thmname{#1}\thmnumber{ #2}\thmnote{#3}}
\theoremstyle{example}
\newtheorem{xexample}[equation]{Example}
\newenvironment{example}{ \begin{myleftbar} \begin{xexample} }{ \end{xexample}\end{myleftbar} }

\newtheoremstyle{examplecont}
  {\topsep} {\topsep}%
  {\upshape}
  {}
  {\bfseries}
  {.}
  {\parindent}
  {\thmname{#1}\thmnumber{ #2} \thmnote{#3}\enspace(continued)}
\theoremstyle{examplecont}
\newtheorem*{xexamplecont}{Example}
\newenvironment{examplecont}{ \begin{myleftbar} \begin{xexamplecont} }{ \end{xexamplecont}\end{myleftbar} }

\newfloat{alg}{tbp}{alg} \floatname{alg}{Algorithm}
\newcommand{\tmpalgoidx}{}
\newcommand{\tmpalgoname}{}
\newcommand{\tmpalgocap}{}
\newcommand{\tmpalgolbl}{}
\newenvironment{algorithm}[4][def]{%
  \ifthenelse{\equal{#1}{def}}{\renewcommand{\tmpalgoidx}{#2}}{\renewcommand{\tmpalgoidx}{#1}}%
  \renewcommand{\tmpalgoname}{#2}%
  \renewcommand{\tmpalgocap}{#3}%
  \renewcommand{\tmpalgolbl}{#4}%
  \index{\tmpalgoname@{\sc \tmpalgoidx}}
  \begin{alg}%
  \begin{tabbing}%
  \quad\quad\=\quad\=\quad\=\quad\=\quad\=\quad\=\quad\=\kill%
  {\sc \tmpalgoname} \\
}{%
  \end{tabbing}%
  \caption{\tmpalgocap}\label{alg_\tmpalgolbl}%
  \end{alg} %
}
\newcounter{algln}
\newcommand{\lnreset}{\setcounter{algln}{0}}
\newcommand{\lnp}{\addtocounter{algln}{1} {\footnotesize \arabic{algln}}}

\newcommand{\compdataHeaderPic}[2]
{	
	\begin{minipage}{5mm} 
		(#1) \\
	\end{minipage} 
	&
	\begin{minipage}{20mm}
		\begin{tabular}{c}
			\includegraphics[width=18mm]{#2} \cr
			\vspace{1mm}
		\end{tabular}
	\end{minipage} 
}
\newcommand{\compdataData}[6]
{
	\begin{minipage}{45mm}
		\begin{tabular}{ll}
			$X$:         & $#1$ \\
			$\dim(L)$:   & $#2$ \\
			$L/\Rad(L)$: & \ifthenelse{\equal{#3}{0}}{#3}{$#3$-dim\ifthenelse{\equal{#4}{}}{}{: $#4$}} \\
			runtime:    & #5 \\
			\ifthenelse{\equal{#6}{}}{}{\multicolumn{2}{l}{\emph{#6}} \\}
		\end{tabular}
	\end{minipage}
}
\newcommand{\compdata}[8]
{	
	\compdataHeaderPic{#1}{#2} &
	\compdataData{#3}{#4}{#5}{#6}{#7}{#8}
}

\newcommand{\Magma}{{\sc Magma}}
\newcommand{\mcF}{\mathcal F}
\newcommand{\mcI}{\mathcal I}
\newcommand{\mcFstar}{{\mcF}^*}
\newcommand{\chr}{\operatorname{char}}
\newcommand{\mcA}{\mathcal A}
\newcommand{\mcL}{\mathcal L}
\newcommand{\mcM}{\mathcal M}
\newcommand{\mcR}{\mathcal R}
\newcommand{\F}{\mathbb{F}}
\newcommand{\bfF}{\mathbf{F}}
\newcommand{\bff}{\mathbf{f}}
\newcommand{\axbc}{a^{\bfF}_{xbc}}
\newcommand{\Rad}{\operatorname{Rad}}
\newcommand{\rk}{\operatorname{rk}}

\newcommand{\lbXdX}[2]{[#1,\ldots,#2]}
\newcommand{\lbXXdXX}[4]{[#1,[#2,\ldots,[#3,#4]]]}
\newcommand{\lbXXdXXX}[5]{[#1,[#2,\ldots,[#3,[#4,#5]]]]}

\title{On Lie Algebras Generated by Few Extremal Elements}
\author{Dan Roozemond, University of Sydney}

\begin{document}
\maketitle

\begin{abstract}
We give an overview of some properties of Lie algebras generated by at most $5$ extremal elements. 
In particular, for any finite graph $\Gamma$ and any field $K$ of characteristic not $2$, we consider an algebraic variety $X$ over $K$ whose $K$-points parametrize Lie algebras generated by extremal elements. Here the generators correspond to the vertices of the graph, and we prescribe commutation relations corresponding to the nonedges of $\Gamma$.

We show that, for all connected undirected finite graphs on at most $5$ vertices, $X$ is a finite-dimensional affine space. Furthermore, we show that for maximal-dimensional Lie algebras generated by $5$ extremal elements, $X$ is a single point. The latter result implies that the bilinear map describing extremality must be identically zero, so that all extremal elements are sandwich elements and the only Lie algebra of this dimension that occurs is nilpotent.

These results were obtained by extensive computations with the {\Magma} computational algebra system.
The algorithms developed can be applied to arbitrary $\Gamma$ (i.e.,~without restriction on the number of vertices), and may be of independent interest.
\end{abstract}

\section{Introduction}
We assume throughout the paper that $K$ is a field of characteristic distinct from $2$.
Let $L$ be a Lie algebra over $K$. A non-zero element $x \in L$ is called \emph{extremal} if $[x,[x,y]] \in K x$ for all $y \in L$. 
if $x$ is extremal, the existence of a linear map $f_x : L \rightarrow K$ such that $[x,[x,y]] = f_x(y)x$ for all $y \in L$ immediately follows from linearity of $[ \cdot, \cdot ]$. If for some extremal element $x \in L$ this linear map $f_x$ is identically $0$, we call $x$ a \emph{sandwich}.

Extremal elements were originally introduced by Chernousov \cite{Che89} in his proof of the Hasse principle for $\mathrm E_8$. {Zel'manov} and Kostrikin proved that, for every $n$, the universal Lie algebra $L_n$ generated by a finite number of sandwich elements $x_1, \ldots, x_n$ is finite-dimensional \cite{ZK90}. Cohen, Steinbach, Ushirobira, and Wales generalized this result and proved that a Lie algebra generated by a finite number of extremal elements is finite dimensional. Moreover, they give an explicit lower bound on the number of extremal elements required to generate each of the classical Lie algebras \cite{CSUW01}. Recently, in 't panhuis, Postma, and the author gave explicit presentations for Lie algebras of type $\mathrm A_n$, $\mathrm B_n$, $\mathrm C_n$, and $\mathrm D_n$, by means of minimal sets of extremal generators \cite{iPR09}.

Moreover, Draisma and in 't panhuis considered finite graphs and corresponding algebraic varieties whose points parametrize Lie algebras generated by extremal elements. They proved in particular that if the graph is a simply laced Dynkin diagram of affine type, all points in an open dense subset of the affine variety parametrize Lie algebras isomorphic to the simple Chevalley Lie algebra corresponding to the associated Dynkin diagram of finite type \cite{DP08}.

Looking at these Lie algebras from a different point of view, Cohen, Ivanyos, and the author proved that if $L$ is a Lie algebra over a field $K$ (of characteristic distinct from $2$ and $3$) that has an extremal element that is not a sandwich, then $L$ is generated by extremal elements, with one exception in characteristic $5$ \cite{CIR08}. 

The strong connection between extremal elements and geometries is further investigated in two papers by Cohen and Ivanyos \cite{CI06,CI07}, and in the Ph.D.~theses by Postma \cite{Postma07} and in 't panhuis \cite{panhuis09}.

\section{Preliminaries and main results}\label{sec_prelim_mainres}
We follow the setup of \cite{iPR09} and \cite{DP08}. Assume that $\Gamma$ is a connected undirected finite graph with $n$ vertices, without loops or multiple bonds, and that $K$ is a field of characteristic distinct from $2$. We let $\Pi$ be the vertex set of $\Gamma$ and denote adjacency of two vertices $x,y \in \Pi$ by $i \sim j$.

We denote by $\mcF(K, \Gamma)$ (often abbreviated to $\mcF$) the quotient of the free Lie algebra over $K$ generated by $\Pi$ modulo the relations
\[
	[x, y] = 0 \mbox{ for all } x,y\in \Pi\mbox{ with } x \not\sim y.
\]
Often, we write elements of $\mcF$ as linear combinations of \emph{monomials} $\lbXXdXX{x_1}{x_2}{x_{l-1}}{x_l}$, where $x_1, \ldots, x_l \in \Pi$. The \emph{length} of such a monomial is said to be $l$, and we often abbreviate such a monomial to $\lbXdX{x_1}{x_l}$. Note that $\mcF$ inherits the natural $\mathbb N$-grading from the free Lie algebra generated by $\Pi$; homogeneous elements of $\mcF$ are linear combinations of monomials of equal length.

We write $\mcFstar$ for the space of all $K$-linear functions $\mcF \rightarrow K$. For every $f \in (\mcFstar)^\Pi$, also written $(f_x)_{x \in \Pi}$, we denote by $\mcL(K, \Gamma, f)$ (often abbreviated to $\mcL(f)$) the quotient of $\mcF(K, \Gamma)$ by the ideal $\mcI(f)$ generated by the infinitely many elements
\[
	[x,[x,y]] - f_x(y)x \mbox{ for } x \in \Pi \mbox{ and } y \in \mcF.
\]

By construction $\mcL(f)$ is a Lie algebra generated by $|{\Pi}| = n$ extremal elements, the extremal generators corresponding to the vertices of $\Gamma$ and commuting whenever they are not adjacent. The element $f_x \in \mcFstar$ is a parameter expressing the extremality of $x \in \Pi$.

In the Lie algebra $\mcL(0)$ the elements of $\Pi$ map to sandwich elements. By \cite{ZK90} this Lie algebra is finite-dimensional; for general $f \in (\mcFstar)^\Pi$ we have $\dim(\mcL(f)) \leq \dim(\mcL(0))$ by \cite[Lemma 4.3]{CSUW01}. It is therefore natural to focus on the Lie algebras $\mcL(f)$ of maximal possible dimension, i.e.,~those of dimension $\dim(\mcL(0))$. We define the set 
\[
	X := \{ f \in (\mcFstar)^\Pi \mid \dim(\mcL(f)) = \dim(\mcL(0)) \},
\]
the parameter space for all maximal-dimensional Lie algebras of the form $\mcL(f)$.

\begin{example}\label{ex_twogen}
	Consider the case where $\Gamma$ consists of two vertices $x$, $y$, connected by an edge.
	Then for every $f = (f_x, f_y) \in (\mcFstar)^\Pi$ the Lie algebra $\mcL(f)$ is spanned by $B = \{ x, y, [x,y] \}$ since $[x,[x,y]] = f_x(y) x$ (where $f_x(y) \in K$), and similarly $[y,[x,y]] = -f_y(x) y$. 
	
	In particular, for $f = (0,0)$ we find that $\mcL(f)$ is the Heisenberg algebra, which is $3$-dimensional. 
	For general $f$, the requirement that $\mcL(f)$ is $3$-dimensional implies (using the Jacobi identity)
	\begin{align*}
		f_y(x) [x,y] = [x, [y, [y, x]&]] = [y,[x,[y,x]]] + [[x,y],[y,x]] = \cr
			& -[y,[x,[x,y]]] + 0 = -f_x(y)[y,x] = f_x(y)[y,x],
	\end{align*}
	so that $f_x(y)$ must be equal to $f_y(x)$ since we assumed $[x,y] \neq 0$.
	
	Consequently, $3$-dimensional Lie algebras generated by a distinguished pair of extremal generators are parametrized by the single value $f_x(y)$.
	As mentioned before, if $f_x(y) = 0$ then $\mcL(f)$ is the Heisenberg algebra.
	It is straightforward to verify that those Lie algebras where $f_x(y) \neq 0$ are mutually isomorphic and isomorphic to the split simple Lie algebra of type $\mathrm A_1$. 	The parameter space $X$ defined above, then, is the affine line, and all Lie algebras corresponding to the nonzero points on the line are mutually isomorphic.
\end{example}

The following theorem asserts that the two generator case is exemplary: $X$ always carries the structure of an affine algebraic variety.
\begin{theorem}[{\cite[Theorem 1]{DP08}}]\label{thm_Kvariety_struct}
	The set $X$ is naturally the set of $K$-rational points of an affine variety of finite type defined over $K$. This variety can be described as follows. Fix any finite-dimensional subspace $V$ of $\mcF$ consisting of homogeneous elements such that $V + \mcI(0) = \mcF$. Then the restriction map
	\[
	X \rightarrow (V^*)^\Pi, \; f \mapsto (f_x\mid_V)_{x \in \Pi}
	\]
	maps $X$ injectively onto the set of $K$-rational points of a closed subvariety of $(V^*)^\Pi$. This yields a $K$-variety structure of $X$ which is independent of the choice of $V$.
\end{theorem}

\begin{examplecont}[\ref{ex_twogen}]
We give an example of an embedding of $X$ as mentioned in Theorem \ref{thm_Kvariety_struct}.
Recall that we showed $X$ to be the affine line; we may take $V = B = \{x,y,[x,y]\}$. 
Then, for $f \in X$ we have 
\[
f \mapsto ((f_x(x), f_x(y), f_x([x,y])), (f_y(x), f_y(y), f_y([x,y])),
\]
and if we again identify elements $f \in X$ with $f_x(y) \in K$ we find $f \mapsto ((0, f, 0),(f,0,0))$.
\end{examplecont}

We have now introduced enough background to be able to state our main results.

\begin{theorem}\label{thm_537_Xtriv}
Let $\Gamma$ be the complete graph on $5$ vertices, $K$ a field of characteristic not $2$, 
and, as above, the set $X$ the parameter space for all maximal-dimensional Lie algebras of the form $\mcL(K, \Gamma, f)$.
Then $X = \{ 0 \}$.
\end{theorem}

The following corollary immediately follows from this theorem and \cite[Lemma 4.2]{CSUW01} (and Section \ref{sec_comp_results} with regards to the dimensions mentioned).
\begin{corollary}\label{cor_537_nilp}
Suppose $L$ is a Lie algebra generated by $5$ extremal elements over the field $K$, where $\chr(K) \neq 2$. If $L$ is of maximal dimension among such Lie algebras (i.e.,~$\dim(L) = 537$ if $\chr(K) \neq 3$ and $\dim(L)=538$ if $\chr(K) = 3$) then $L$ is nilpotent.
\end{corollary}

Looking at results for other cases shows why this result is interesting. For example, if $\Gamma$ is the complete graph on on $2$, $3$, or $4$ vertices, then $X$ is nontrivial. Moreover, for almost all $f \in X$, $\mcL(f)$ is a simple Lie algebra of type $\mathrm A_1$, $\mathrm A_2$, and $\mathrm D_4$, respectively (see \cite{CSUW01,Roo05mt} and Section \ref{sec_comp_results} of this paper). Moreover, there are infinite families of graphs $\Gamma$, for example the affine Dynkin diagrams, where $\mcL(f)$ is almost always a simple Lie algebra of classical type (see \cite{iPR09,DP08}). In short: this is the first nontrivial case encountered where $X$ is a point. 

Note that the fact that $X$ is a point means that the extremality of the generators and the maximal-dimensionality assumption force all generators to be sandwiches, and therefore force $L$ to be nilpotent.

We mention one more result.
\begin{theorem}\label{thm_all_X_affine}
Let $\Gamma$ be a connected undirected finite graph with at most $5$ vertices, without loops or multiple bonds,
$K$ a field of characteristic not $2$, 
and, as above, the set $X$ the parameter space for all maximal-dimensional Lie algebras of the form $\mcL(K, \Gamma, f)$.
Then $X$ is isomorphic to a finite-dimensional affine space over $K$. The dimension of $X$ is given in Tables \ref{tab_comp_results23} -- \ref{tab_comp_results5}.
\end{theorem}

This theorem provides a partial answer to one of the questions posed in \cite[Section 5.2]{DP08}, namely, ``Is $X$ always an affine space?''.
Similarly, the results discussed in Section \ref{sec_comp_results} suggest that in the cases we considered one of the other questions posed, ``Is there always a generic Lie algebra?'' can be answered affirmatively.

\subsection*{The remainder of this paper}
In Section \ref{sec_extr_props} we present a number of known results on these Lie algebras and their elements that we need for our proofs and calculations.
In Section \ref{sec_fsets} we introduce $f$-sets, a concept we need in order to be able to calculate with the affine algebraic variety $X$ that parametrizes the Lie algebras we are interested in.
In Section \ref{sec_algs} we discuss the algorithms we developed and implemented in {\Magma} to obtain the computation results presented in Section \ref{sec_comp_results}. 
Finally, in Section \ref{sec_conclusion} we briefly reflect on our results and possible future research.

\section{Some properties of extremal elements}\label{sec_extr_props}
In this section we present a number of properties of extremal elements and the Lie algebras generated by them.
The following identities go back to Premet and are commonly called the \index{Premet identities}\emph{Premet identities}. 
We have, for $x$ an extremal element and for all $y,z \in L$:
\begin{align}
	2 [x,[y,[x,z]]] & = f_x([y,z])x - f_x(z)[x,y] - f_x(y)[x,z], \label{eqn_premet1} \\
	2 [[x,y],[x,z]] & = f_x([y,z])x + f_x(z)[x,y] - f_x(y)[x,z]. \label{eqn_premet2}
\end{align}

There is a number of properties of the functions $f_x$ introduced earlier that we should discuss. The following result is an important tool in our computations.
\begin{theorem}[{\cite[Theorem 2.5]{CSUW01}}]\label{thm_csuw_bil_f}
Suppose that $L$ is a Lie algebra over $K$ generated by extremal elements.
There is a unique bilinear symmetric form $f: L \times L \rightarrow K$ such that, for each extremal element $x \in L$, the linear form $f_x$ coincides with $y \mapsto f(x,y)$. This form is associative, in the sense that $f(x,[y,z]) = f([x,y],z)$ for all $x,y,z \in L$.
\end{theorem}

This theorem implies a number of identities in the values of $f$, as presented in the following lemma.
We use these relations extensively in our computer calculations.
\begin{lemma}\label{lem_ttr_f}
Let $x,y,z,y_1, \ldots, y_l \in \Pi$, let $(f_x)_{x \in \Pi} \in (\mcF^*)^\Pi$, and let
$m \in \mcL(f)$. Then the following equalities hold:
\begin{enumerate}
	\item\label{itm_lem_ttr_fx_cx}\label{itm_lem_ttr_f_first} $f_x(y) = 0$ whenever $x = y$ or $x \not\sim y$,
	\item\label{itm_lem_ttr_fx_y} $f_x(y) = f_y(x)$,
	\item\label{itm_lem_ttr_fx_cxm} $f_x([y,m]) = 0$ whenever $x = y$ or $x \not\sim y$,
	\item\label{itm_lem_ttr_fx_ym} $f_x([y,m]) = -f_y(x,m)$,
	\item\label{itm_lem_ttr_fx_yzx} $f_x([y,[z,x]]) = f_x(y)f_x(z)$,
	\item\label{itm_lem_ttr_fx_yxm} $f_x([y,[x,m]]) = -f_x(y)f_x(m)$,
	\item\label{itm_lem_ttr_fx_yzm} $f_x([y,[z,m]]) = f_z([y,[x,m]]) - f_z([x,[y,m]])$,
	\item\label{itm_lem_ttr_f_rev}\label{itm_lem_ttr_f_last} $f_x(\lbXXdXX{y_1}{y_2}{y_{l-1}}{y_l}) = 
	(-1)^{l-1} f_{y_l}(\lbXXdXXX{y_{l-1}}{y_{l-2}}{y_2}{y_1}{x})$.
\end{enumerate}
\end{lemma}

\begin{proof}
\ref{itm_lem_ttr_fx_cx} follows from the observation that $[x,[x,y]] = 0$ whenever $x = y$ or $x \not\sim y$.
\ref{itm_lem_ttr_fx_y} follows immediately from the symmetry of $f$ in Theorem \ref{thm_csuw_bil_f}: $f_x(y) = f(x,y) = f(y,x) = f_y(x)$.
\ref{itm_lem_ttr_fx_cxm} and \ref{itm_lem_ttr_fx_ym} follow similarly from its associativity: $f_x([y,m]) = f(x, [y,m]) = f([x,y], m)$, which is equal to $0$ for case \ref{itm_lem_ttr_fx_cxm} and reduces to $-f([y,x], m) = -f(y, [x,m]) = -f_y([x,m])$ in the case of \ref{itm_lem_ttr_fx_ym}.
For \ref{itm_lem_ttr_fx_yzx}, observe $f_x([y,[z,x]]) = f(x,[y,[z,x]]) = -f(y,[x,[z,x]]) = f(y,f_x(z)x) = f_x(z)f(y,x) = f_x(z)f_x(y)$ by associativity and bilinearity of $f$; \ref{itm_lem_ttr_fx_yxm} follows similarly.
\ref{itm_lem_ttr_fx_yzm} follows from the Jacobi identity and associativity: $f_x([y,[z,m]]) = f(x,[y,[z,m]]) = -f([x,y],[m,z]) = -f([[x,y], m], z) = 
-f_z([[x,y],m]) = -f_z([x,[y,m]]) + f_z([y,[x,m]])$.
Finally, to see \ref{itm_lem_ttr_f_rev} simply apply associativity of $f$ and anti-symmetry of Lie algebras $l-1$ times.
\end{proof}

\section{$f$-sets}\label{sec_fsets}
Our goal in the research described here was to find more information on the structure of the variety $X$ describing the parameter space, 
and of $\mcL(f)$ for various $f \in X$. 
In order to achieve this goal, we have to compute a multiplication table for $\mcL(f)$. There are two ways to approach this problem:
either pick an $f \in X$ in advance and compute a multiplication table for $\mcL(f)$, or compute a general multiplication table, i.e.,~one that has entries in the coordinate ring of $X$. The latter approach, which we have chosen to pursue,  has two important advantages: Firstly, we only need to compute the multiplication table once, and can afterwards easily instantiate it for particular $f \in X$. Secondly, by insisting that $\mcL(f)$ be a Lie algebra and of maximal dimension we obtain information about the structure of $X$. 

\begin{examplecont}[\ref{ex_twogen}]
In this example we would automatically recover the relation $f_x(y) = f_y(x)$ via the Jacobi identity on $y$, $x$, and $[x,y]$:
\[
0 = [y,[x,[x,y]]] + [x,[[x,y],y]] + [[x,y],[y,x]] = f_x(y)[y,x] + f_y(x)[x,y] + 0,
\]
and $f_x([x,y]) = 0$ via evaluating $[x,[x,[x,y]]]$ from the inside out:
\[
f_x([x,y])x = [x,[x,[x,y]]] = [x, f_x(y) x] = f_x(y) [x,x] = 0,
\]
so that $f_x([x,y]) = 0$ since the fact that $\mcL(f)$ is of maximal dimension implies that $x \neq 0$.
\end{examplecont}

To aid in the description of our algorithms we introduce the concept of \emph{$f$-sets}.
For the remainder of this section fix a graph $\Gamma$ and a field $K$, 
and let $B$ be a set of monomials of $\mcF(K, \Gamma)$ that project to a monomial $K$-basis of $\mcL(0)$
(which exists by \cite[Lemma 3.1]{iPR09}).
Furthermore, let $\bfF$ be some subset of $\Pi \times B$, and let $R_\bfF$ be the rank $|\bfF|$ multivariate polynomial ring $K[\bff_y(c) \mid (y,c) \in \bfF]$ with variables $\bff_y(c)$.
Furthermore, let $r : \Pi \times B \rightarrow R_\bfF$, $(x,b) \mapsto r_x(b)$ be a map from $\Pi \times B$ into this polynomial ring.
We call such a pair $(\bfF, r)$ an \emph{$f$-set of size $|\bfF|$}.

Moreover, if $r$ is such that for all $f \in X$ we have, under the evaluation $\bff_y(c) \mapsto f_y(c)$ for $(y,c) \in \bfF$,
\[
r_{x}(b) = f_{x}(b) \;\;\; \mbox{ for all } x \in \Pi, b \in B,
\]
then we call $(\bfF, r)$ a \emph{sufficient $f$-set}.

Finally, suppose $(\bfF, r)$ is an $f$-set where we require, to exclude trivialities, that $r_y(c) = \bff_y(c)$ for all $(y,c) \in \bfF$.
We construct, for an arbitrary $v \in K^{\bfF}$, an element $(f_{x})_{x \in \Pi}$ of $(\mcF^*)^\Pi$ as follows. Take, for every $b \in B$, the value $f_{x}(b) \in K$ to be $r_{x}(b) \in R_\bfF$ evaluated in $v$, that is, $\bff_y(c) \mapsto v_{(y,c)}$. This defines a map $\phi: K^{\bfF} \rightarrow (\mcF^*)^\Pi$. Note that $\mcL(\phi(v))$ is always a Lie algebra, but not always one of maximal dimension. If, however, $\mcL(\phi(v))$ is a Lie algebra of dimension $\dim(\mcL(0))$ for all $v \in K^{\bfF}$, then we call $(\bfF, r)$ a \emph{free $f$-set}.

\bigskip

For example, $(\Pi \times B, (y,c) \mapsto \bff_{y}(c))$ is trivially a sufficient $f$-set. It is, however, in general not a free $f$-set. On the other hand, $(\emptyset, (y,c) \mapsto 0)$ forms a free $f$-set, but in general not a sufficient one.

\begin{examplecont}[\ref{ex_twogen}]
Recall $\Pi = \{ x, y\}$ and $B = \{ x, y, [x,y] \}$, so that $\Pi \times B = \{ (x,x), (x,y), (x,[x,y]), (y,x), (y,y), (y,[x,y]) \}$. 
Then $(\Pi \times B, (y,c) \mapsto \bff_y(c))$ is a (trivial) sufficient $f$-set of size $6$.
In order for $\mcL(f)$ to be a Lie algebra of maximal dimension $3$ we must have $f_x(x) = f_x([x,y]) = f_y(y) = f_y([x,y]) = 0$, and we have seen before that $f_y(x)$ must be equal to $f_x(y)$.
This means that $(\bfF, r)$ is not a free $f$-set, since $\phi(v)$ is not a Lie algebra if $v_{(x,y)} \neq v_{(y, x)}$.

However, if we take $\bfF = \{ (x,y) \}$, $r_x(y) = r_y(x) = \bff_x(y)$, and $r_x(x) = r_x([x,y]) = r_y(y) = r_y(x,y) = 0$,
then $(\bfF, r)$ is a sufficient $f$-set of size $1$, 
and by checking the Jacobi identity for $\mcL(f)$ we find that it is a free sufficient $f$-set.
\end{examplecont}

\begin{lemma}\label{lem_free_fset_implies_affine_space}
If $\bfF$ is a free sufficient $f$-set then $X \cong K^\bfF$.
\end{lemma}
\begin{proof}
Recall from above the map $\phi : K^\bfF \rightarrow (\mcF^*)^\Pi$. By the assumption that $\bfF$ is free, we have $\phi(K^\bfF) \subseteq X$. The assumption that $\bfF$ is sufficient guarantees that $X$ is not a proper subset of $\phi(K^\bfF)$, so that $\phi$ is a bijection between $K^\bfF$ and $X$.
\end{proof}

We remark that a priori it is possible for $X$ to be an affine space while no free sufficient $f$-set exists. However, we have not yet encountered such a case (see also the discussion in Section \ref{sec_conclusion}).

\section{The algorithms used}\label{sec_algs}
In this section we describe the algorithms used to obtain the results in Section \ref{sec_comp_results}. 
For the computational results in this paper we have developed an algorithm that,
given a graph $\Gamma$ and a field $K$, computes the variety $X$ and the multiplication
table of $\mcL(f)$ for any $f \in X$. 
These algorithms were implemented in the {\Magma} computer algebra system \cite{Magma217}.
The algorithm may be broken up into four distinct parts, of which the third ({\sc MultiplicationTable}) and the fourth ({\sc MinimizeFSet}) are the most time-consuming.

\subsection{Computing a basis for $\mcL(0)$}\label{sec_alg_comp_basis}

\begin{algorithm}{ComputeBasis}{Computing a basis}{comp_basis}
{\bf in:} \>\>\> A field $K$ and a graph $\Gamma$, with $\Pi = V(\Gamma)$. \\
{\bf out:} \>\>\> A set of monomial elements of $\mcF$. \\
\textbf{begin} \lnreset \\
\lnp \> \textbf{let} $U$ be the free associative algebra with generators $\Pi$, \\
\lnp \> \textbf{let} $\mu: \mcF \rightarrow U$ be the linear map defined by $x \mapsto x$ for $x \in \Pi$ and \\
     \> \> $[b,c] \mapsto \mu(b)\mu(c) - \mu(c)\mu(b)$ for $b, c \in \mcL(0)$, \\
\lnp \> \textbf{let} $B = B' = \Pi \subseteq \mcF$ and $B_U = B_U' = \{ \mu(x) \mid x \in B \} \subseteq U$, \\
\lnp \> \textbf{let} $U \supseteq I_U = \langle x^2 \mid x \in \Pi \rangle \cup \langle \mu([x,y]) \mid x,y \in \Pi, x \not\sim y \rangle$,  \\
\lnp \> \textbf{while} $|B_U'| \neq 0$ \textbf{ do } \\
     \> \> \emph{//Update the ideal representing extremality of elements of $\Pi$} \\
\lnp \> \> \textbf{let} $I_U = I_U \cup \langle \mu([x,[x,b]]) \mid x \in \Pi, b \in B' \rangle$, \\
     \> \> \emph{//Find new elements that are not linear combinations of existing basis elements} \\
\lnp \> \> \textbf{let} $B'_0 = B' \subseteq \mcF$, $B' = \emptyset  \subseteq \mcF$, $B_U' = \emptyset  \subseteq U$, \\
\lnp \> \> \textbf{for each} $x \in \Pi, b \in B'_0$ \textbf{do} \\
\lnp \> \> \> \textbf{let} $p = \mu([x,b])$, \\
\lnp \> \> \> \textbf{if} $p \notin KB_U + I_U$ \textbf{then} \\
\lnp \> \> \> \> \textbf{let} $B = B \cup [x,b]$, $B' = B' \cup [x,b]$, \\
\lnp \> \> \> \> \textbf{let} $B_U = B_U \cup \{p\}$, $B_U' = B_U' \cup \{p\}$ \\
\lnp \> \> \> \textbf{end if}, \\
\lnp \> \> \textbf{end for}, \\
\lnp \> \textbf{end while}, \\
\lnp \> \textbf{return} $B$.\\
\textbf{end}
\end{algorithm}

In this section we describe how we compute, strictly speaking, a set of monomial elements of $\mcF(K, \Gamma)$ 
that projects to a basis of $\mcL(f)$ for any $f \in X$. However, because in particular such a basis projects to
a basis $\mcL(0)$ that consists of monomials, we will often call such a basis ``a monomial basis of $\mcL(0)$'' in the
remainder.

A sketch of the algorithm {\sc ComputeBasis} is given as Algorithm \ref{alg_comp_basis}.
We initialize $B$ with the set $\Pi$ of generators of $\mcF(K, \Gamma)$, so that $B$ contains a basis for 
all monomials of length $1$. We then iteratively extend a basis for monomials up to length
$l$ to a basis for monomials up to length $l+1$ by forming all products $[x, b]$, where $x \in \Pi$
and $b$ a monomial of length $l$, and testing whether this product can be written in previously
found basis elements. Once we have arrived at a length for which no new basis elements are found,
we are finished.

Note that while we return elements of $\mcF$, all nontrivial computations take place inside the
free associative algebra $U$ on the sandwich generators of a Lie algebra $L_0$ (which will be identical to $\mcL(0)$).
In order to connect the two algebras we introduce the map $\mu$ in line $2$, that maps elements of $\mcF$ to elements of $U$. 
In the course of the algorithm, we are also constructing an ideal $I_U$ of $U$ such that 
the adjoint representation $U \rightarrow \operatorname{End}(L_0)$ factors through $U \rightarrow U/I_U$. 
Here $I_U$ describes the fact that the generators are sandwich elements; $I_U$ should be viewed as the 
analog in $U$ of the ideal $\mcI(0)$ of $\mcF$ (cf. Section \ref{sec_prelim_mainres}).
This construction is what allows us to include $x^2$ in $I_U$ in line $4$ of the algorithm, 
since $\operatorname{ad}_x^2 = 0$ if $x$ is a sandwich.

Note finally that in order to be 
able to compute in $U/I_U$ (as required in line $9$ of Algorithm \ref{alg_comp_basis}) 
we have to repeatedly compute a Gr\"obner basis of $I_U$ but, fortunately, 
since the elements of $I_U$ are all homogeneous a truncated Gr\"obner basis is sufficient
(and sufficiently efficient in the cases we are interested in).

\begin{lemma}\label{lem_alg_compute_basis}
The algorithm {\sc ComputeBasis} returns a set $B$ of monomials in $\mcF$, and $B$ projects to a basis of $\mcL(f)$ for any $f \in X$. 
Moreover, all elements of $B$ are either of the form $x$ ($x \in \Pi$) or of the form $[x,b]$ ($x \in \Pi$, $b \in B$).
\end{lemma}
\begin{proof}
The fact that $B$ consists of monomials, and that these monomials are of the stated form, 
is immediate from the algorithm (note the construction of new elements of $B$ in line $11$).
Moreover, because our $f \in X$ are such that $\dim(\mcL(f)) = \dim(\mcL(0))$, 
a homogeneous basis for $\mcL(0)$ is a homogeneous basis for $\mcL(f)$ (cf.~\cite[Lemma 3.1]{iPR09}.).
\end{proof}

\subsection{The initial sufficient $f$-set}\label{sec_alg_init_fset}

\begin{algorithm}{InitialFSet}{Finding an initial $f$-set}{init_fset}
{\bf in:} \>\>\> A field $K$ and a graph $\Gamma$, with $\Pi = V(\Gamma)$, and a monomial basis $B$ of $\mcL(0)$. \\
{\bf out:} \>\>\> A sufficient $f$-set $(\bfF, r)$. \\
\textbf{begin} \lnreset \\
\lnp \> \textbf{let} $\bfF = \emptyset$, \\
\lnp \> \textbf{for} $c \in B$, $y \in \Pi$ \textbf{do} \\
\lnp \> \> \textbf{if} $f_y(c)$ can be expressed in $\{ f_{x}(b) \mid (x, b) < (y, c)\}$ using 
	Lemma \ref{lem_ttr_f}\ref{itm_lem_ttr_f_first} -- \ref{itm_lem_ttr_f_last} \textbf{then} \\
\lnp \> \> \> \textbf{set} $r_y(c) \in R_\bfF$ to be the corresponding expression, where $f_x(b)$ is replaced by $r_x(b)$\\
\lnp \> \> \textbf{else} \\
\lnp \> \> \> \textbf{set} $\bfF = \bfF \cup \{(y,c)\}$, \\
\lnp \> \> \> \textbf{set} $r_y(c) = \bff_y(c)$ \\
\lnp \> \> \textbf{end if} \\
\lnp \> \textbf{end for} \\
\lnp \> \textbf{return} $(\bfF, r)$.\\
\textbf{end}
\end{algorithm}

The computation of the multiplication table, described in Section \ref{sec_alg_comp_mt}, will take place over the ring $R_\bfF$ for some sufficient $f$-set $(\bfF, r)$. In this section we describe {\sc InitialFSet}, the procedure for initialisation of $R_\bfF$, as presented in Algorithm \ref{alg_init_fset}. Note that the $f$-set returned by this algorithm is a sufficient $f$-set, but in general not a free $f$-set.

We fix in advance an (arbitrary) total ordering on the elements of $\Pi$, and a total ordering on $B$ respecting the natural order by monomial length, i.e.,~$b < c$ whenever $l(b) < l(c)$. These extend to a total ordering of $\Pi \times B$ by $(x, b) < (y, c)$ whenever $b < c$, or $b = c$ and $x < y$. The main for loop in Algorithm \ref{alg_init_fset} traverses $\Pi \times B$ in ascending order; this is the ordering meant by ``$<$'' in line $3$ as well.
{\sc InitialFSet} uses the relations in Lemma \ref{lem_ttr_f} as reduction steps: $A = B$ is interpreted as ``$A$ may be reduced to $B$''. 

\begin{lemma}\label{lem_alg_initial_fset}
The algorithm {\sc InitialFSet} returns a sufficient $f$-set.
\end{lemma}
\begin{proof}
Observe that the relations in Lemma \ref{lem_ttr_f} all must hold if $\mcL(f)$ is to be a Lie algebra of dimension $\mcL(0)$; this immediately implies sufficiency. Finally, observe that all reductions thus obtained are polynomial, so that indeed $r_y(c)$ defined in line $4$ of Algorithm \ref{alg_init_fset} is an element of $R_\bfF$.
\end{proof}

\subsection{Computing the multiplication table}\label{sec_alg_comp_mt}

\begin{algorithm}{MonomialToBasis}{Expressing a monomial as a linear combination of basis elements}{monomial_to_basis}
{\bf in:} \>\>\> A field $K$ and a graph $\Gamma$, with $\Pi = V(\Gamma)$, a monomial basis $B$ of $\mcL(0)$, \\
\>\>\> a sufficient $f$-set $(\bfF, r)$,  a partially defined multiplication table $(\axbc)$, \\
\>\>\> and a monomial $m = \lbXdX{x_1}{x_l}$. \\
{\bf out:} \>\>\> $v \in (R_\bfF)^B$ such that $m = \sum_{b \in B} v_b b$, or \textbf{fail}. \\
\textbf{begin} \lnreset \\
\lnp \> \textbf{if} $m$ can be written as a linear combination $v \in (R_\bff)^B$ of basis elements, \\
     \> by iteratively computing $\lbXdX{x_i}{x_l}$ for $i=l-1,\ldots,1$ using $(\axbc)$ \textbf{then} \\
\lnp \> \> \textbf{return} $v$. \\
\lnp \> \textbf{else if} $x_{l-1} \not\sim x_l$ in $\Gamma$ \textbf{then} \\
\lnp \> \> \textbf{return} $0 \in (R_\bfF)^B$. \\
\lnp \> \textbf{else if} $x_i = x_{i+1}$ for some $i < l$ \textbf{then} \\
     \> \> \emph{/* Use $[x,[x,y]] = f_x(y) x$ */} \\
\lnp \> \> \textbf{let} $x = x_i$, $y = \lbXdX{x_{i+2}}{x_l}$, \\
\lnp \> \> \textbf{let} $t = r_x(y) \in R_\bfF$, \\
\lnp \> \> \textbf{let} $v$ such that $\sum v_b b = \lbXdX{x_1}{x_i}$ (using $(\axbc)$), \\
\lnp \> \> \textbf{return} $t \cdot v$. \\
\lnp \> \textbf{else if} $x_i = x_{i+2}$ for some $i<l-2$ \textbf{then} \\
     \> \> \emph{/* Use $2[x,[y,[x,z]]] = f_x([y,z]) x - f_x(z)[x,y] - f_x(y)[x,z]$ */} \\
\lnp \> \> \textbf{let} $x = x_i$, $y = x_{i+1}$, $z = \lbXdX{x_{i+3}}{x_l}$, \\
\lnp \> \> \textbf{let} $t_{xyz} = r_x([x_{i+1},z]) \in R_\bfF$, \\
\lnp \> \> \textbf{let} $t_{xy} = r_x(y) \in R_\bfF$, \\
\lnp \> \> \textbf{let} $t_{xz} = r_x(z) \in R_\bfF$, \\
\lnp \> \> \textbf{let} $v^{x}$  such that $\sum v^{x}_b b = \lbXdX{x_1}{x_i}$ (using $(\axbc)$), \\
\lnp \> \> \textbf{let} $v^{xy}$ such that $\sum v^{xy}_b b = \lbXXdXXX{x_1}{x_2}{x_{i-1}}{x_i}{y}$ (using $(\axbc)$), \\
\lnp \> \> \textbf{let} $v^{xz}$ such that $\sum v^{xz}_b b = \lbXXdXXX{x_1}{x_2}{x_{i-1}}{x_i}{z}$ (using $(\axbc)$), \\
\lnp \> \> \textbf{return} $\frac{1}{2}t_{xyz} \cdot v^{x} -\frac{1}{2} t_{xz} \cdot v^{xy} -\frac{1}{2} t_{xy} \cdot v^{xz}$. \\
\lnp \> \textbf{else} \\
\lnp \> \> \textbf{return fail} \\
\lnp \> \textbf{end if}. \\
\textbf{end}
\end{algorithm}

\begin{algorithm}{MonomialRelations}{Finding monomial relations}{monomial_relations}
{\bf in:} \>\>\> A field $K$ and a graph $\Gamma$, with $\Pi = V(\Gamma)$, a monomial basis $B$ of $\mcL(0)$, \\
\>\>\> a sufficient $f$-set $(\bfF, r)$, a partially defined multiplication table $(\axbc)$, \\
\>\>\> a set $\mcM$ of monomials of length $l$ currently under consideration, \\
\>\>\> an integer $k \geq 1$, and a monomial $m = \lbXdX{x_1}{x_l} \in \mcM$. \\
{\bf out:} \>\>\> A set $\mcA$ of elements $a \in (R_\bfF)^{B \times \mcM}$ such that $m = \sum_{b \in B} a_b b + \sum_{u \in \mcM} a_u u$. \\
\textbf{begin} \lnreset \\
\lnp \> \textbf{let} $\mcA = \emptyset$, \\
     \> \emph{/* By the Jacobi identity */} \\
\lnp \> \textbf{write} $\lbXdX{x_k}{x_l}$ as a sum $s$ of $2^{l-k+4}$ monomials using Equation (\ref{eqn_jac_from_start}), \\
\lnp \> \textbf{find} $a \in (R_\bfF)^{B \times \mcM}$ such that $\lbXXdXX{x_1}{x_2}{x_{k-1}}{s} = \sum_{b \in B} a_b b + \sum_{u \in \mcM} a_u u$,  \\
     \> using {\sc MonomialToBasis} to obtain the entries indexed by $B$, \\
\lnp \> \textbf{let} $\mcA = \mcA \cup {a}$, \\
     \> \emph{/* Using Lemma \ref{lem_eqn_SxNxM} */} \\
\lnp \> \textbf{for} $(i,j)$ such that $x_i = x_j$, $1 \leq i \leq k$, and $i+3 \leq j \leq l$ \textbf{do} \\
\lnp \> \> \textbf{write} $m$ as a sum $s$ of $i-j+1$ monomials using Lemma \ref{lem_eqn_SxNxM}, \\
\lnp \> \> \textbf{find} $a \in (R_\bfF)^{B \times \mcM}$ such that $s = \sum_{b \in B} a_b b + \sum_{u \in \mcM} a_u u$, using \\
     \> \> {\sc MonomialToBasis} for the entries indexed by $B$, \\
\lnp \> \> \textbf{let} $\mcA = \mcA \cup {a}$ \\
\lnp \> \textbf{end for}, \\
\lnp \> \textbf{return} $\mcA$. \\
\textbf{end}
\end{algorithm}

\begin{algorithm}{ComputeMultiplicationTable}{Computing a multiplication table}{compute_mt}
{\bf in:} \>\>\> A field $K$ and a graph $\Gamma$, with $\Pi = V(\Gamma)$, a monomial basis $B$ of $\mcL(0)$, \\
\>\>\> and a sufficient $f$-set $(\bfF, r)$. \\
{\bf out:} \>\>\> A partial multiplication table $(\axbc)$, where $(x,b,c) \in \Pi \times B \times B$. \\
\textbf{begin} \lnreset \\
\lnp \> \textbf{let} $n = \max_{b \in B}|b|$, \\
\lnp \> \textbf{for} $l = 1, \ldots, n$ \textbf{do} \\
     \> \> \emph{/* Consider monomials of length $l$ */} \\
\lnp \> \> \textbf{let} $\mcM = \{ [x,b] \mid x \in \Pi, b \in B \mbox{ such that } |b|=l\}$, \\
\lnp \> \> \textbf{let} $M^L$ be the empty $0 \times |\mcM|$ matrix, \\
\lnp \> \> \textbf{let} $M^R$ be the empty $0 \times |B|$ matrix, \\
\lnp \> \> \textbf{let} $k = 1$, \\
\lnp \> \> \textbf{while} $\rk(M^L) < |\mcM|$ \textbf{do} \\
\lnp \> \> \> \textbf{for} $m \in \mcM$ \textbf{do} \\
     \> \> \> \> \emph{/* Find relations for $m$ */} \\
\lnp \> \> \> \> \textbf{if} $v =$ {\sc MonomialToBasis}($m$) $\neq$ \textbf{fail} \textbf{then} \\
\lnp \> \> \> \> \> \textbf{append} $e_m$ to $M^L$, \textbf{append} $v$ to $M^R$ \\
\lnp \> \> \> \> \textbf{else} \\
\lnp \> \> \> \> \> \textbf{let} $\mcA =$ {\sc MonomialRelations}$(\mcM, k, m)$, \\
\lnp \> \> \> \> \> \textbf{for each} $a \in \mcA$ \textbf{append} $e_m-a|_\mcM$ to $M^L$, \textbf{append} $a|_B$ to $M^R$ \\
\lnp \> \> \> \> \textbf{end if}, \\
\lnp \> \> \> \textbf{end for}, \\
\lnp \> \> \> \textbf{let} $k = k + 1$, \\
\lnp \> \> \textbf{end while}, \\
     \> \> \emph{/* Read the result */} \\
\lnp \> \> \textbf{for} $m = [x,b] \in \mcM$ \textbf{do} \\
\lnp \> \> \> \textbf{let} $v = w M^R$, where $w$ is such that $w M^L = e_m$, \\
\lnp \> \> \> \textbf{let} $\axbc = v_c$ \textbf{for} $c \in B$ \\
\lnp \> \> \textbf{end for} \\
\lnp \> \textbf{end for}, \\
\lnp \> \textbf{return} $(\axbc)$. \\
\textbf{end}
\end{algorithm}

The third part of the algorithm is {\sc MultiplicationTable}, where part of a multiplication table for $\mcL(f)$ is computed. 
Let $(\bfF, r)$ be a sufficient $f$-set (such as the one produced by {\sc InitialFSet}), and let $\Pi$ be the set of generating extremal elements and $B$ be a monomial basis of $\mcL(0)$ as before.
We seek to compute $\axbc \in R_\bfF$ such that, for $x,b \in B$, for specialisation to any $f \in X$, we have
\[
[x, b] = \sum_{c \in B} \axbc c.
\]
We will call such a set $(\axbc)$ a \emph{multiplication table of $\mcL(f)$}.

Note that this is a multiplication table over $R_\bfF$; for $f \in X$ it specializes to a multiplication table for the Lie algebra $\mcL(f)$ over $K$ upon evaluation $\bff_y(c) \mapsto f_y(c)$ for $(y,c) \in \bfF$.
First, observe that once we have computed $\axbc$ for all $x \in \Pi$ and all $b \in B$ every monomial $[x_1, \ldots, x_l]$ can easily be written as a linear combination of basis elements by using $(\axbc)$ to multiply ``from the inside out''. Secondly, observe that then similarly for any $(b,c) \in B \times B$ we can use the Jacobi identity to write $[b,c]$ as a linear combination of monomials, and then apply the same trick (it would of course in general yield an exponential number of terms). Therefore we restrict to computing the $(\axbc)$ for $x \in \Pi$, $b \in B$.

We present sketches of the algorithms used as Algorithms \ref{alg_monomial_to_basis}, \ref{alg_monomial_relations}, and \ref{alg_compute_mt}, where \ref{alg_compute_mt} is the main function. We clarify some of the notation: Firstly, we consider elements of $(R_\bfF)^B$ as vectors indexed by elements of $B$, as for example in lines $8$, $15$--$17$ of Algorithm \ref{alg_monomial_to_basis}. These vectors may of course be multiplied by elements of $R_\bfF$, as in lines $9$, $12$--$14$ of that algorithm. The same notation is used for elements of $(R_\bfF)^\mcM$ and $(R_\bfF)^{B \times \mcM}$. Secondly, we write $e_m$ for the vector with $0$ everywhere except for the position indexed by $m$ (cf.~line $10$ of Algorithm \ref{alg_compute_mt}), and $a|_S$ for the restriction of the vector $a$ to a subset $S$ of its index set. 

To clarify the monomial rewriting in line $6$ of Algorithm \ref{alg_monomial_relations}, 
we state the following lemma.
\begin{lemma}\label{lem_eqn_SxNxM}
Let $m = \lbXdX{x_1}{x_l}$ be a monomial of length $l$, 
such that $x_i = x_j$ for some $i,j$ such that $j-i \geq 3$.
Then $m$ may be rewritten as a sum of $3$ monomials of length less than $l$ and $j-i-2$ monomials of length equal to $l$.
\end{lemma}
For brevity and ease of reading we will prove the following easier lemma, of which the above is a straightforward generalisation.
\begin{lemma}
Let $m = \lbXdX{x_1}{x_l}$ be a monomial of length $l$, 
such that $x_1 = x_j$ for some $j$ such that $j \geq 4$.
Then $m$ may be rewritten as a sum of $3$ monomials of length less than $l$ and $j-3$ monomials of length equal to $l$.
\end{lemma}
\begin{proof}
Observe that we have the following equality of a monomial of length $l \geq 3$ and a sum of different monomials, 
obtained by repeatedly applying the Jacobi identity:
\begin{align}\label{eqn_jac_from_start}
[x_1, [x_2, [ x_3, \ldots, x_l]]] =\; & [x_2, [x_1, [x_3, \ldots, x_l]]] - [[x_3, \ldots, x_l], [x_1, x_2]]  \cr
=\; & [x_2, [x_1, [x_3, \ldots, x_l]]] - [x_3, [ [ x_4, x_5, \ldots, x_l], [x_1, x_2]]] \cr
& + [x_4, [ [ x_3, x_5, \ldots, x_l], [x_1, x_2]]] \cr
=\; & {\ldots}\;(2^{l-3}\mbox{ terms in total}).
\end{align}

Moreover, it follows immediately from the Jacobi identity that
\begin{align}\label{eqn_bring_together1}
[x_1, [x_2, [ x_3, \ldots, x_l]]] = - [[x_2, x_1], [ x_3, \ldots, x_l]] + [x_2, [x_1, [ x_3, \ldots, x_l]]]
\end{align}

Observe that there are $i-2$ generators between $x_1$ and $x_i$, so that by applying Equation (\ref{eqn_bring_together1}) $i-3$ times 
(but starting at $x_2$ rather than $x_1$), we find
\begin{align}\label{eqn_SxNxM_1}
[x_1, [x_2, [ x_3, \ldots, x_l]]] =\; & [x_1, [[x_{i-1}, \ldots, x_2], [x_i, [x_{i+1}, \ldots, x_l]]]] + [i-3 \mbox{ monomials of length } l],
\end{align}
By Equation (\ref{eqn_premet1}) the first monomial reduces to monomials of smaller length (recall $x_1 = x_i$):
\begin{align}\label{eqn_SxNxM_2}
[x_1, [[x_{i-1}, \ldots, x_2], [x_i, [x_{i+1}, \ldots, x_l]]]] =\;& \frac{1}{2} f_{x_1}([x_{i-1}, \ldots, x_2], [x_{i+1}, \ldots, x_l]) x_1 \cr
& -\frac{1}{2}f_{x_1}([x_{i+1}, \ldots, x_l]) [x_1, [x_{i-1}, \ldots, x_2]] \cr
& - \frac{1}{2} f_{x_1}([x_{i-1}, \ldots, x_2]) [x_1, [x_{i+1}, \ldots, x_l]],
\end{align}
so that, combining Equations (\ref{eqn_SxNxM_1}) and (\ref{eqn_SxNxM_2}), we find
\begin{align}\label{eqn_SxNxM_3}
[x_1, [x_2, [ x_3, \ldots, x_l]]] =\;& [3 \mbox{ monomials of length} < l] + [i-3 \mbox{ monomials of length } l], 
\end{align}
proving the lemma.
\end{proof}

\begin{lemma}\label{lem_alg_compute_mt}
If {\sc MultiplicationTable} returns a multiplication table $(\axbc)$ then, for any $f \in X$, these $(\axbc)$ specialize to a multiplication table for the Lie algebra $\mcL(f)$ upon evaluation $\bff_y(c) \mapsto f_y(c)$ for $(y,c) \in \bfF$.
\end{lemma}
\begin{proof}
Recall that, as described at the start of Section \ref{sec_alg_comp_mt}, the structure constants $(\axbc)$ for $x,b,c \in B$ follow immediately from those where $x \in \Pi$ and $b,c \in B$. We will consequently assume that $(\axbc)$ is known for $x,b,c \in B$.

Observe that in Algorithm \ref{alg_compute_mt} the products are computed by ascending length, 
and for each length $l$ two matrices $M^L$ and $M^R$ are computed. 
Throughout the algorithm the rows of $M^L$ and $M^R$ contain linear relations between elements of $\mcM$ and elements of $B$, i.e.,~the $i$-th row of $M^L$  must be equal to the $i$-th row of $M^R$.
To see that these relations are true, observe that those gained from {\sc MonomialToBasis} (Algorithm \ref{alg_monomial_to_basis}) are true by induction on the length of $m$ (if returned in line $2$), by definition of extremality (if returned in line $9$), or by the first Premet identity Equation \ref{eqn_premet1} (if returned in line $18$). Similarly, those gained from {\sc MonomialRelations} (Algorithm \ref{alg_monomial_relations})  are true by the Jacobi identity (those produced in line $3$) or by Lemma \ref{lem_eqn_SxNxM} (those produced in line $8$).

Also, because we traverse the products of basis elements by ascending length, the elements of $(R_\bfF)^B$ required in lines $8$ and $15$--$17$ of Algorithm \ref{alg_monomial_to_basis} are guaranteed to be easily deduced from the entries of $(\axbc)$ computed earlier.
Finally, correctness of the $\axbc$ computed by solving the system of linear equations represented by $M^L$ and $M^R$ in line $18$--$21$ of Algorithm \ref{alg_compute_mt} follows from bilinearity of the Lie algebra multiplication.
\end{proof}

We remark that we cannot give an estimate of the number of relations we have to produce (i.e.,~the number of times we call {\sc MonomialRelations} and the value of $k$) before $M^L$ is of full rank. 
In practice, however, we find that $k$ does generally not exceed $2$.
In particular, this implies that termination of {\sc ComputeMultiplicationTable} is not guaranteed -- but by Lemma \ref{lem_alg_compute_mt} correctness upon termination is. 

\subsection{Finding a minimal sufficient $f$-set}\label{sec_alg_minimize_fset}

\begin{algorithm}{MinimizeFSet}{Finding a minimal sufficient $f$-set}{minimize_fset}
{\bf in:} \>\>\> A field $K$ and a graph $\Gamma$, with $\Pi = V(\Gamma)$, a monomial basis $B$ of $\mcL(0)$,   \\
\>\>\> a sufficient $f$-set $(\bfF, r)$, and a multiplication table $(\axbc)$. \\
{\bf out:} \>\>\> An updated $f$-set $(\bfF, r)$ of smaller size, and updated multiplication table. \\
\textbf{begin} \lnreset \\
     \> \emph{/* Collect relations between elements of $\bfF$ */} \\
\lnp \> \textbf{let} $\mcR = \{ [y,[y,c]] - f_y(c)y \mid (y,c) \in \bfF \}$, \\
\lnp \> \textbf{let} $\mcR = \mcR \cup \{ [a,[b,c]] + [b,[c,a]] + [a,[b,c]] \mid a,b,c \in B \}$, \\
     \> \emph{/* Reduce size of $\bfF$ using $\mcR$ */} \\
\lnp \> \textbf{do} \\
\lnp \> \> \textbf{for each} non-zero coefficient $t \in R_{\bfF}$ of each element $u \in \mcR$ \textbf{do} \\
\lnp \> \> \> \textbf{find} $\alpha \in K^*$, and $(y,c)$ such that $\bff_y(c)$ occurs only in the linear term $\alpha \bff_y(c)$ of $t$ \\
\lnp \> \> \> \> \textbf{or continue} if no such element exists, \\
\lnp \> \> \> \textbf{set} $r_y(c) = -\frac{1}{\alpha}(t - \alpha \bff_y(c))$, \\
\lnp \> \> \> \textbf{replace} $\bff_y(c)$ by $r_y(c)$ in $\mcR$, \\
\lnp \> \> \> \textbf{replace} $\bff_y(c)$ by $r_y(c)$ in $(\axbc)$, \\
\lnp \> \> \> \textbf{set} $\bfF = \bfF \backslash \{ (y,c) \}$ \\
\lnp \> \> \textbf{end for} \\
\lnp \> \textbf{until} $\bfF$ is unchanged, \\
\lnp \> \textbf{return} $(\bfF, r), (\axbc)$. \\
\textbf{end}
\end{algorithm}

After the multiplication table has been computed, we try to minimize the size of the $f$-set. This procedure is called {\sc MinimizeFSet} and presented here as Algorithm \ref{alg_minimize_fset}. The procedure is straightforward: we find a set of relations $\mcR$ between the $\bff_y(c)$ by requiring that the Jacobi identity holds and that for all $y \in \Pi$, $c \in B$ we have $[y,[y,c]] = f_y(c) y$, and we subsequently use those relations to find nontrivial $r_y(c)$ and remove $(y,c)$ from $\bfF$. 

Once the algorithm finishes, we apparently could not find any new linear relations. This does not necessarily mean that the $f$-set $(\bfF, r)$ returned is free. However, if it happens to be the case that $\mcR = \{ 0 \}$ upon exiting, then $(\bfF, r)$ is indeed a free sufficient $f$-set since then the Jacobi identity apparently holds. For all the cases we have tested, the $f$-set returned was in fact free.

Observe that in this algorithm only division by elements of $K^*$ is taking place (cf. lines $5$--$6$ of Algorithm \ref{alg_minimize_fset}) and not by other elements of $R_\bfF$, so that all these calculations take place in multivariate polynomial rings $R_\bfF$ over $K$. While it is a priori conceivable that division by elements of $R_\bfF \backslash K^*$ would be necessary to reduce the size of the $f$-set (and thus that our algorithm would return a non-free $f$-set), we have not observed such a case in practice.

\bigskip

We remark that a rather more sophisticated alternative to Algorithm \ref{alg_minimize_fset} has been implemented.
For example, we execute a similar algorithm in the course of the computation of the multiplication table (rather than only after the fact),
we alternate between collecting a relatively small number relations and reducing the size of $\bfF$ (rather than do them in succession),
and in line $5$ of Algorithm \ref{alg_minimize_fset} we give preference to $y,c$ such that the pair $(y,c)$ is large in the ordering on $\Pi \times B$ introduced in Section \ref{sec_alg_init_fset}.

\begin{lemma}\label{lem_alg_minimize_fset}
If $\mcR = \{ 0 \}$ upon exiting of the algorithm {\sc MinimizeFSet}, then $X$ is isomorphic to the affine space $K^{|\bfF|}$.
\end{lemma}
\begin{proof}
This follows immediately from Lemma \ref{lem_free_fset_implies_affine_space}.
\end{proof}

\section{Computational results}\label{sec_comp_results}

\begin{table}
\hspace{-10mm}
\begin{tabular}{cclccl}
\compdata{a}{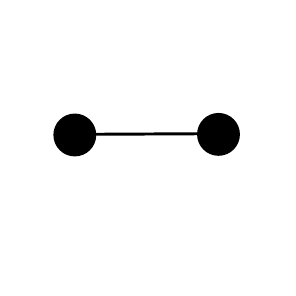}{K^1}{3}{3}{\mathrm A_1}{0s}{}   \cr
\compdata{b}{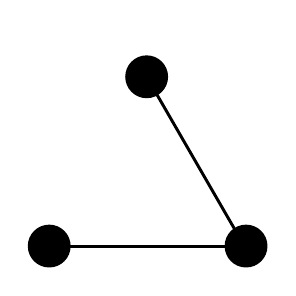}{K^2}{6}{3}{\mathrm A_1}{0s}{}  &
\compdata{c}{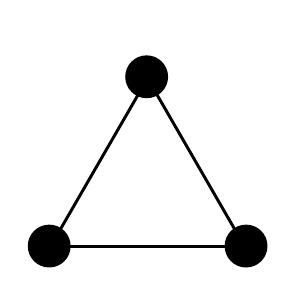}{K^4}{8}{8}{\mathrm A_2}{0s}{}  
\end{tabular}
\caption{Computational results (2 or 3 generators)}\label{tab_comp_results23}
\end{table}

\begin{table}
\hspace{-10mm}
\begin{tabular}{cclccl}
\compdata{a}{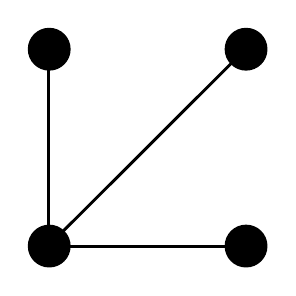}{K^3}{12}{3}{\mathrm A_1}{0s}{}   &
\compdata{d}{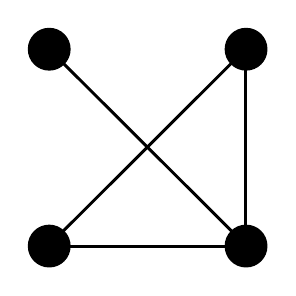}{K^5}{15}{15}{\mathrm A_3}{0s}{}  \cr
\compdata{b}{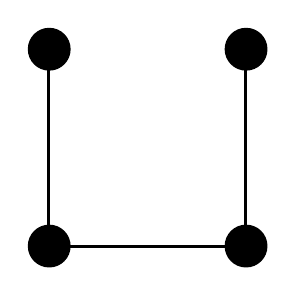}{K^3}{10}{10}{\mathrm B_2}{0s}{}  &
\compdata{e}{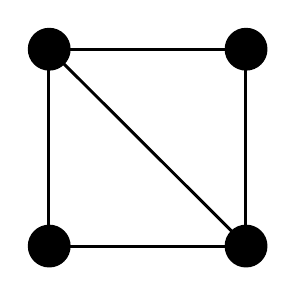}{K^8}{21}{21}{\mathrm B_3}{0s}{}  \cr
\compdata{c}{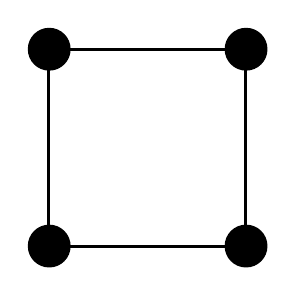}{K^5}{15}{15}{\mathrm A_3}{0s}{}  &
\compdata{f}{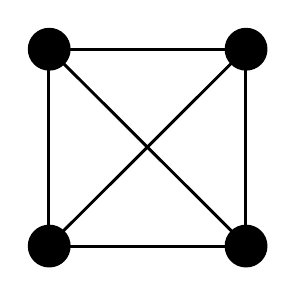}{K^{12}}{28}{28}{\mathrm D_4}{0s}{} 
\end{tabular}
\caption{Computational results (4 generators)}\label{tab_comp_results4}
\end{table}

\begin{table}
	\hspace{-10mm}
\begin{tabular}{cclccl}
 \compdata{a}{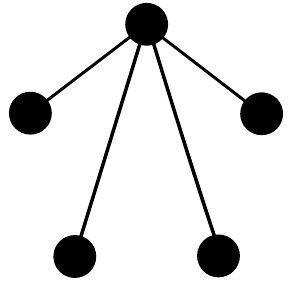}{K^{  5}}{ 28}{28}{\mathrm D_4}{0s}{} &
\compdata{g}{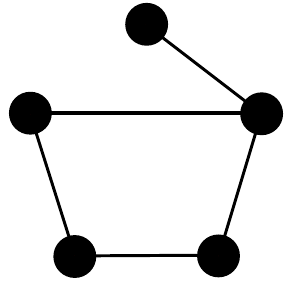}{K^{  6}}{ 30}{15}{\mathrm A_3}{0s}{} \cr
 \compdata{b}{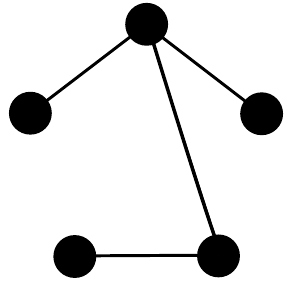}{K^{  4}}{ 20}{10}{\mathrm B_2}{0s}{} &
 \compdata{h}{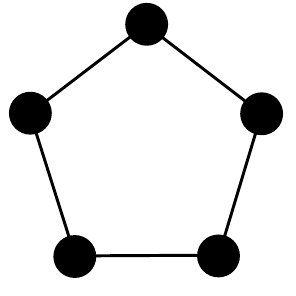}{K^{  6}}{ 24}{24}{\mathrm A_4}{0s}{} \cr
 \compdata{c}{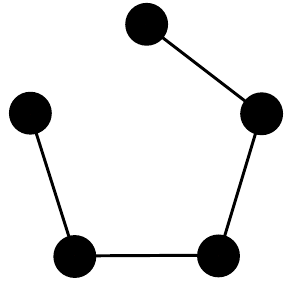}{K^{  4}}{ 15}{10}{\mathrm B_2}{0s}{} &
 \compdata{i}{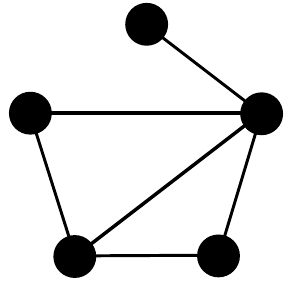}{K^{ 10}}{ 52}{52}{\mathrm F_4}{0s}{} \cr
 \compdata{d}{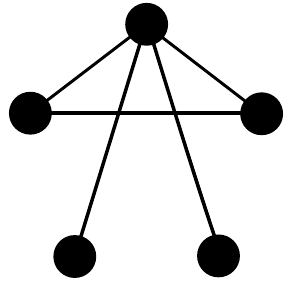}{K^{  7}}{ 36}{36}{\mathrm B_4}{0s}{} &
 \compdata{j}{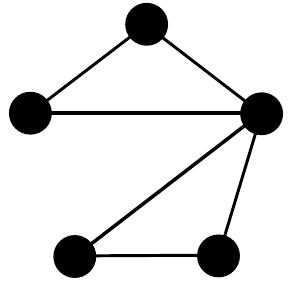}{K^{  9}}{ 45}{45}{\mathrm D_5}{0s}{} \cr
 \compdata{e}{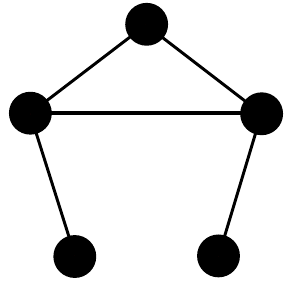}{K^{  6}}{ 30}{15}{\mathrm A_3}{0s}{} &
 \compdata{k}{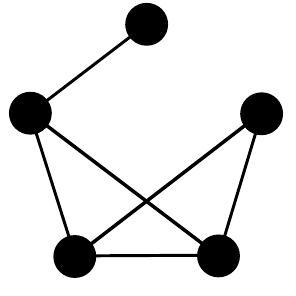}{K^{  9}}{ 45}{45}{\mathrm D_5}{0s}{} \cr
\compdata{f}{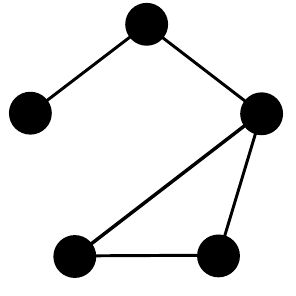}{K^{  6}}{ 24}{24}{\mathrm A_4}{0s}{} &
\compdata{l}{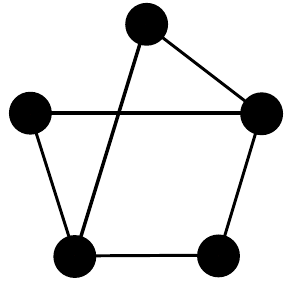}{K^{ 10}}{ 52}{52}{\mathrm F_4}{0s}{} 
\end{tabular}
\caption{Computational results (5 generators) (1/2)}
\end{table}
\addtocounter{table}{-1}

\begin{table}
	\hspace{-10mm}
\begin{tabular}{cclccl}
\compdata{m}{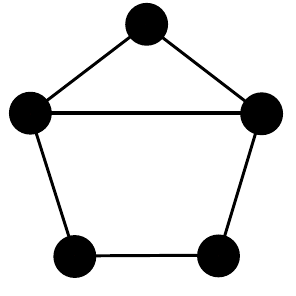}{K^{  9}}{ 45}{45}{\mathrm D_5}{0s}{}    &
 \compdata{r}{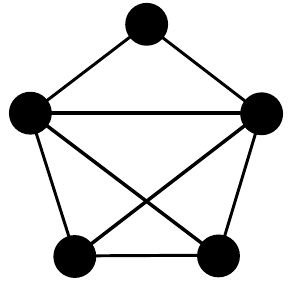}{K^{ 12}}{134}{28}{\mathrm D_4}{10s}{}   \cr
 \compdata{n}{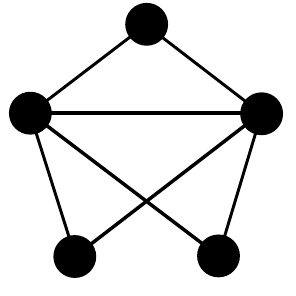}{K^{ 13}}{ 86}{28}{\mathrm D_4}{3s}{}    &
 \compdata{s}{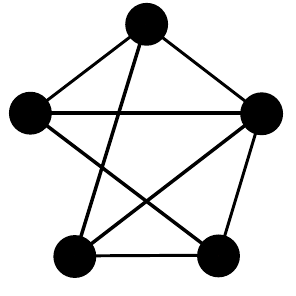}{K^{ 21}}{133}{133}{\mathrm E_7}{14s}{}  \cr
 \compdata{o}{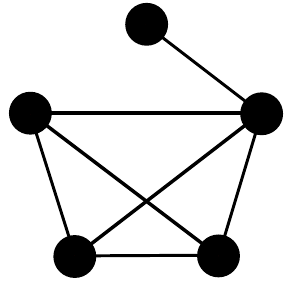}{K^{ 14}}{ 78}{78}{\mathrm E_6}{2s}{}    &
\compdata{t}{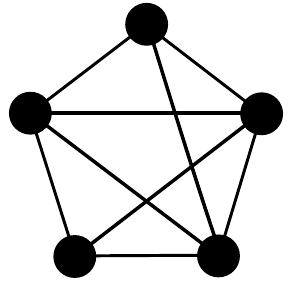}{ K^{21}}{249}{78}{\mathrm E_6}{2510s}{} \cr   
\compdata{p}{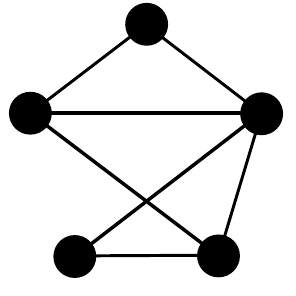}{ K^{14}}{ 78}{78}{\mathrm E_6}{2s}{}    &
\compdataHeaderPic{u}{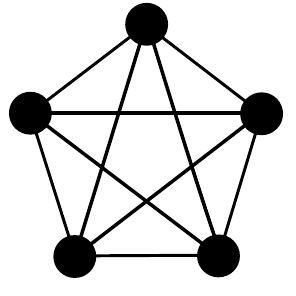} &
	\begin{minipage}{45mm}
		\begin{tabular}{ll}
			$X$:         & $\{0\}$ \\
			\multicolumn{2}{l}{$\dim(L): \left\{ \begin{array}{rl} 538 & \mbox{if }\chr(K) = 3 \\ 537 & \mbox{otherwise} \end{array} \right.$} \\
			$L/\Rad(L)$: & \mbox{trivial} \\
			runtime:    & 38260s 
		\end{tabular}
	\end{minipage} \cr
\compdata{q}{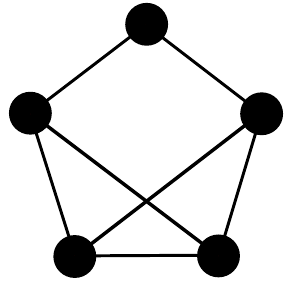}{ K^{14}}{ 78}{78}{\mathrm E_6}{2s}{}
\end{tabular}
\caption{Computational results (5 generators) (2/2)}\label{tab_comp_results5}
\end{table}

We present the results of applying the algorithms described in Section \ref{sec_algs} in Tables \ref{tab_comp_results23} -- \ref{tab_comp_results5}.
For each connected undirected finite graph $\Gamma$ with $n$ vertices, $2 \leq n \leq 5$, without loops or multiple bonds, we state the dimension of the variety $X$, the dimension of $\mcL(0)$, the generic dimension and isomorphism type of $L/\Rad(L)$, and the CPU time taken (on a $2$ GHz AMD Opteron) to execute the steps described in Section \ref{sec_algs}.

Recall that we claim the results presented in Tables \ref{tab_comp_results23} -- \ref{tab_comp_results5} to be valid for any field $K$ of characteristic distinct from $2$. To establish this, we have performed the following calculations for every graph $\Gamma$ occurring in the results. First we calculate a basis of the equivalent Lie ring over $\mathbb Z$ using a procedure similar to the one described in \cite[Section 3]{CdG2009}. This gives us a set $\mathcal P_1$ of primes for which the monomial basis computed in Algorithm \ref{alg_comp_basis} would differ from the characteristic $0$ case. Then we execute the algorithms described in Section \ref{sec_algs} for $K = \mathbb Q$, and in Algorithm \ref{alg_minimize_fset} we store in the set $\mathcal P_2$ the prime components of the denominator of the $\alpha$ we divide by. Subsequently, we perform the computation with $K = \mathrm{GF}(p)$ for $p \in \mathcal (P_1 \cup \mathcal P_2) \backslash \{2\}$; fortunately this last task is easily done in parallel. Only in one case did we find a different result: for $\Gamma = \mathrm{K}_5$ the case where $\chr(K)=3$ is of dimension $538$ (rather than $537$). The result that $X = \{ 0 \}$ remains valid, however.

We emphasize that in each of the cases considered we found a free $f$-set, and consequently prove that $X$ is an affine space (cf.~Lemma \ref{lem_free_fset_implies_affine_space}). Moreover, we keep track of the reductions and relations used to turn the trivial sufficient $f$-set into a free sufficient $f$-set, so that we obtain a certificate for the correctness of these results.
Unfortunately, however, such certificates are far too long to reproduce here.

Before turning to some of the results, let us elaborate on a particular piece of data, namely the isomorphism type of $L/\Rad(L)$. For a number of special $\Gamma$ (namely the Dynkin diagrams of affine type) it has been shown that the Lie algebra $\mcL(f)$ is isomorphic to a fixed Lie algebra $L$ for $f$ in an open dense subset of $X$ \cite[Theorem 22]{DP08}. A similar result is proved in \cite[Section 7]{iPR09} for the four infinite families there considered. We investigate this property in our cases as follows. Firstly, we carry out all the computations described in the previous section over $\mathbb{Q}$. Recall that in all cases we found a free $f$-set $(\bfF, r)$; we let $l = |\bfF|$, so that $X \cong \mathbb{Q}^l$. Now take some finite field $\F_q$, whose characteristic $p$ is distinct from $2$ and such that no multiples of $p$ occur in denominators in the multiplication table $(\axbc)$ (in our examples the multiplication table always turns out to be integral, so any finite field of characteristic distinct from $2$ works).
Then take $v$ uniformly random in $\F_q^l$ and let $L = \mcL(\phi(v))$, where we interpret the multiplication table $(\axbc)$ over $R_\bfF \otimes \F_q$ and $\phi$ is evaluation as described in Section \ref{sec_fsets}. This gives us a Lie algebra $L$ defined over $\F_q$, for which it is straightforward for {\Magma} to compute the isomorphism type of $L/\Rad(L)$. 
We repeated this process a number of times over various finite fields (e.g.,~$\F_{17}$, $\F_{101}$) and stored the results -- in each case finding the same dimension and isomorphism type of $L/\Rad(L)$ in the vast majority of cases. We remark that we have not distinguished between split and twisted forms of Lie algebras of the same Cartan type.

The following results are particularly worth pointing out (we refer to the case presented in Table i(x) simply as case [ix]).
\begin{enumerate}
	\item{} [1c], [2c], [3a], and [3h] agree with cases proved in \cite{DP08};
	\item{} [2b], [2d], [3f], and [3j] agree with cases proved in \cite{iPR09};
	\item{} [3n]--[3q] are the four cases with $5$ vertices and $3$ commuting edges. Where one might expect a structural similarity, [3n] is the odd one out both in terms of dimension ($86$ vs $78$) and type (a large radical vs simple of type $\mathrm E_6$);
	\item{} Similarly, [3r] and [3s] are the two cases with $5$ vertices and $2$ commuting edges. Again, despite the difference in $\dim(\mcL(0))$ being only $1$, their isomorphism types are very different;
	\item{} In our computations, we have not encountered a Lie algebra of type $\mathrm E_8$. Even though case [3t] is of dimension $249$, the biggest simple component we found there was of type $\mathrm E_6$. It has been proved that $\mathrm E_8$ can be generated by $5$ extremal elements \cite[Theorem 8.2]{CSUW01}, but apparently not in the manner considered in this paper: $\mathrm E_8$ is not a subquotient of any maximal-dimensional Lie algebra generated by $5$ extremal elements plus commutation relations, for then it would have to occur either in [3t], where it does not, or in [3u], where it cannot for such a Lie algebra is always nilpotent. On the other hand, $\mathrm E_8$ can be generated in this manner by $9$ extremal elements \cite[Theorem 22]{DP08}.
	\item{} [3u] is the case of largest dimension among the $5$ generators; it is also the only case where $X$ is a point and consequently only one Lie algebra occurs, and that it is a nilpotent Lie algebra (cf.~\cite[Lemma 4.2]{CSUW01}).
\end{enumerate}

\section{Conclusion}\label{sec_conclusion}

\begin{proof}[Proof of Theorems \ref{thm_537_Xtriv} and \ref{thm_all_X_affine}.]
	These theorems follow immediately from the correctness of our algorithm, as proved in Lemmas 
	\ref{lem_alg_compute_basis}, \ref{lem_alg_initial_fset}, \ref{lem_alg_compute_mt}, and \ref{lem_alg_minimize_fset}, 
	and the computational results presented in Section \ref{sec_comp_results}.
\end{proof}

In Theorem \ref{thm_537_Xtriv} and Corollary \ref{cor_537_nilp} we prove that if $\Gamma$ is the complete graph on $5$ vertices, then $X = \{ 0 \}$ and the only Lie algebra that occurs is nilpotent. This of course leads to the question whether this pattern continues: ``For $\Gamma$ the complete graph on $n$ vertices ($n \geq 5$), is $X$ a point?''. Unfortunately, this question is not feasible using our software (experiments show that for $n=6$ already $\dim(\mcL(0)) > 140000$). 
Another approach would be to consider relations between the variety for $\Gamma$ and the varieties for subgraphs $\Gamma' \subseteq \Gamma$. This problem is the subject of ongoing research.

Secondly, as mentioned earlier, two questions were posed in \cite[Section 5.2]{DP08}: ``Is $X$ always an affine space?'', and ``Is there always a generic Lie algebra?''. Although the authors of that paper expect the answers to both questions to be negative, we have found no counterexamples. Moreover, since these questions can be answered affirmatively for both a number of infinite series of sparse graphs (as studied in \cite{DP08,iPR09}) and a number of dense graphs (as studied here), we would not dispose of these conjectures too easily.

\section*{Acknowledgements}
The author would like to thank Arjeh Cohen and Jan Draisma for numerous fruitful discussions on this topic, Willem de Graaf for his help in computing several Lie rings, and the anonymous reviewers for their very thorough evaluation and their invaluable comments.

\def\cprime{$'$}

\end{document}